\documentclass[reqno]{amsart}

\usepackage[utf8]{inputenc}
\usepackage[T1]{fontenc}
\usepackage{lmodern}
\usepackage[top=3cm, bottom=3cm, left=3cm, right=3cm]{geometry}
\usepackage{amsmath, amssymb, amsthm}
\usepackage{hyperref}
\usepackage{xcolor}
\usepackage{cancel}
\usepackage[normalem]{ulem}
\usepackage{enumitem}
\usepackage{fancyhdr}
\usepackage{booktabs}

\newtheorem{theorem}{Theorem}[section]
\newtheorem{definition}[theorem]{Definition}
\newtheorem{corollary}[theorem]{Corollary}
\newtheorem{proposition}[theorem]{Proposition}
\newtheorem{lemma}[theorem]{Lemma}

\newtheorem{remark}[theorem]{Remark}

\newtheorem{example}[theorem]{Example}

\newcommand{\R}{\mathbb R}

\newcommand{\g}{\mathfrak{g}}

\newcommand{\brd}[2]{\left[#1,#2\right]_{\cdot}}
\newcommand{\brc}[2]{\left[#1,#2\right]_{\circ}}

\DeclareMathOperator{\Inn}{Inn}

\DeclareMathOperator{\ad}{ad}

\DeclareMathOperator{\Lie}{Lie}
\DeclareMathOperator{\Aut}{Aut}

\DeclareMathOperator{\Aff}{Aff}
\DeclareMathOperator{\GL}{GL}

\DeclareMathOperator{\Der}{Der}
\DeclareMathOperator{\ev}{ev}

\newcommand{\tcdot}{\mathchoice
  {\raise-0.20ex\hbox{$\displaystyle\tilde{\cdot}$}}
  {\raise-0.20ex\hbox{$\textstyle\tilde{\cdot}$}}
  {\raise-0.15ex\hbox{$\scriptstyle\tilde{\cdot}$}}
  {\raise-0.10ex\hbox{$\scriptscriptstyle\tilde{\cdot}$}}}
\newcommand{\tcirc}{\mathchoice
  {\raise-0.20ex\hbox{$\displaystyle\tilde{\circ}$}}
  {\raise-0.20ex\hbox{$\textstyle\tilde{\circ}$}}
  {\raise-0.15ex\hbox{$\scriptstyle\tilde{\circ}$}}
  {\raise-0.10ex\hbox{$\scriptscriptstyle\tilde{\circ}$}}}
  
  \newcommand{\id}{\mathrm{id}}

\begin{document}

\title[On simple compact Lie Skew Braces]{On simple compact Lie Skew Braces}

\author{Marco Damele}
\address{(Marco Damele) Dipartimento di Matematica \\
         Universit\`a di Cagliari (Italy)}
        \email{m.damele4@studenti.unica.it}

\author{Andrea Loi}
\address{(Andrea Loi) Dipartimento di Matematica \\
         Universit\`a di Cagliari (Italy)}
         \email{loi@unica.it}

\thanks{
The authors are supported by INdAM and  GNSAGA - Gruppo Nazionale per le Strutture Algebriche, Geometriche e le loro Applicazioni;   the research was also supported by ProBiki of Fondazione di Serdegna.}

\subjclass[2000]{16T25, 17B05, 17B30, 22E60, 22E46.} 
\keywords{Lie skew brace; skew left brace, post-Lie group; post-Lie algebra, integration of post Lie algebras, lambda-action, 
solvable and simple Lie skew braces.}

\begin{abstract}
We study simplicity for Lie skew braces from both the global and
infinitesimal points of view. After recalling the correspondence
between connected Lie skew braces, simply transitive affine actions,
and post-Lie algebras, we investigate ideals 
and rigidity phenomena.
Our main result concerns compact connected Lie skew braces. We prove
that a compact connected simple Lie skew brace is either the trivial
Lie skew brace on \(S^1\), or else both underlying Lie groups are
simple and the brace is trivial or almost trivial. In particular,
outside the exceptional case of \(S^1\), simplicity of a compact
connected Lie skew brace is equivalent to simplicity of either
underlying Lie group.
We also show that a connected compact solvable Lie skew brace is
necessarily trivial. Finally, we exhibit a noncompact example showing
that this rigidity fails in general: there exists a connected simply
connected simple Lie skew brace whose additive and multiplicative Lie
groups are both solvable.
\end{abstract}

\maketitle

\tableofcontents  
\section{Introduction}\label{sec:intro}

Skew braces, introduced by Guarnieri and Vendramin
\cite{Guarnieri2017}, form a natural algebraic framework for the study
of set-theoretic solutions of the Yang--Baxter equation. Over the last
few years they have also revealed deep connections with group theory,
Hopf--Galois structures, and solvability phenomena; see, for instance,
\cite{Guarnieri2017,Smoktunowicz2018,Bachiller2016Counterexample,Byott2024,BBEJP}.
In parallel, several Lie-theoretic structures related to Yang--Baxter
theory have proved to be equally fruitful, especially through their
links with affine actions on Lie groups and post-Lie algebra theory.

A smooth analogue of a skew brace is obtained by equipping a single
smooth manifold with two compatible Lie group structures. This leads
to the notion of a \emph{Lie skew brace}. As shown in
\cite{DameleLoi2026}, connected Lie skew braces admit a geometric
interpretation in terms of simply transitive affine actions on Lie
groups; see Theorem~\ref{thm:LSB-affine-correspondence}. On the
infinitesimal side, they determine post-Lie algebra structures in the
sense of \cite{BaiGuoShengTang2023,Burde2012affineaction}. Thus Lie
skew braces provide a natural global framework in which one may study
post-Lie structures together with their integrability.

The present paper is concerned with \emph{simplicity} for Lie skew
braces. More precisely, we study connected Lie skew braces having no
nontrivial proper connected closed ideals; see
Definition~\ref{def:simple-LSB}. 
In the finite theory, simplicity is known to exhibit striking and
sometimes unexpected behaviour: simple skew braces may have solvable,
or even abelian, additive group, and need not arise from simple groups
\cite{Bachiller2016Counterexample,Byott2024simple}. We return to this
comparison in Remark~\ref{rem:finite-simple-skew-braces-solvable}.
A guiding problem
for this work is to determine to what extent analogous phenomena
persist in the Lie setting, and in particular how the answer depends
on compactness.
Our results show that compactness is the decisive dividing line: the
compact connected case is highly rigid, whereas the noncompact case
still admits unexpectedly flexible simple examples; see
Example~\ref{ex:A1-1-model}.


Our first rigidity result deals with the case where the multiplicative
Lie group is simple. Using a theorem of Burde, Dekimpe, and Vercammen
on post-Lie algebra structures \cite{Burde2012affineaction}, we prove
in Theorem~\ref{thm:simple-group-rigidity} that if \((G,\circ)\) is a
simple Lie group, then every Lie skew brace \((G,\cdot,\circ)\) is
necessarily trivial or almost trivial. Equivalently, the two Lie group
laws either coincide or differ only by passage to the opposite group
law.

Our main result concerns the compact connected case; see
Theorem~\ref{thm:main-compact-simple}.

\medskip

\noindent\textbf{Main Theorem.}
\emph{Let \((G,\cdot,\circ)\) be a compact connected Lie skew brace,
and assume that it is simple. Then exactly one of the following holds:}
\begin{enumerate}[label=\textup{(\alph*)}]
    \item \emph{\((G,\cdot)=(G,\circ)=S^1\);}
    \item \emph{\((G,\cdot)\) and \((G,\circ)\) are simple Lie groups,
    and}
    \[
    \circ=\cdot
    \qquad\text{or}\qquad
    \circ=\cdot^{\mathrm{opp}}.
    \]
\end{enumerate}

\medskip

The proof of Theorem~\ref{thm:main-compact-simple} is genuinely
global and does not follow formally from the infinitesimal theory
alone. After passing to the simply connected case, we show that in the
compact semisimple setting the lambda-action is necessarily inner.
This allows us to embed the multiplicative group into a product group
as a subgroup complementary to the diagonal. Passing to Lie algebras,
we obtain a decomposition to which a proposition of Ozeki
\cite{Ozeki1977} applies. This ultimately forces the additive group to
be simple. At that point, the rigidity result proved earlier implies
that the brace is trivial or almost trivial.

As an immediate consequence, outside the exceptional case of \(S^1\),
simplicity of a compact connected Lie skew brace is equivalent to
simplicity of either  of its underlying Lie groups; see
Corollary~\ref{cor:compact-simple-equivalence}. This shows that the
compact connected Lie setting is substantially more rigid than the
finite one. We also discuss solvability in this framework and prove
that a connected compact solvable Lie skew brace is necessarily
trivial.

We conclude by showing that this rigidity is genuinely specific to the
compact case. Indeed, in Example~\ref{ex:A1-1-model} we construct an
explicit connected simply connected simple Lie skew brace whose
additive and multiplicative Lie groups are both solvable. Thus, in
sharp contrast with the compact connected setting, simplicity of a Lie
skew brace does not force simplicity of either underlying Lie group in
the noncompact case.

The paper is organized as follows. In Section~\ref{secLSB} we develop
the general theory of Lie skew braces, affine actions, post-Lie
algebras, ideals, and simplicity, and we prove the rigidity result for
Lie skew braces with simple multiplicative group. In
Section~\ref{sec:compact-simple} we establish the compact connected
rigidity theorem and its consequences, and we conclude with
Example~\ref{ex:A1-1-model}, which shows  that such rigidity fails outside
the compact setting.

\section{Preliminaries}\label{secLSB}

\subsection{Lie skew braces}

A \emph{Lie skew brace} (LSB) is a triple $(G,\cdot,\circ)$, where
$(G,\cdot)$ and $(G,\circ)$ are real Lie groups on the same smooth manifold,
such that
\begin{equation}\label{eqLSB}
a\circ (b\cdot c)=(a\circ b)\cdot a^{-1}\cdot (a\circ c)
\end{equation}
for all $a,b,c\in G$, where $a^{-1}$ denotes the inverse of $a$ in
$(G,\cdot)$. Setting $b=c=e_{\cdot}$ in \eqref{eqLSB}, one sees immediately
that the two group laws have the same identity element, which we denote by $e$.

One can therefore define the category $\mathbf{LSB}$, whose objects are Lie
skew braces and whose morphisms
\[
\varphi:(G_1,\cdot_1,\circ_1)\longrightarrow (G_2,\cdot_2,\circ_2)
\]
are smooth maps such that
\[
\varphi(a\cdot_1 b)=\varphi(a)\cdot_2 \varphi(b),
\qquad
\varphi(a\circ_1 b)=\varphi(a)\circ_2 \varphi(b)
\]
for all $a,b\in G_1$. Thus $\mathbf{LSB}$ is the Lie-theoretic analogue of the
category of skew braces introduced in \cite{Guarnieri2017}; see also
\cite{Smoktunowicz2018}. In particular, when the underlying manifold is
discrete, Lie skew braces reduce to ordinary skew braces.

\medskip

Let $(G,\cdot)$ be a connected real Lie group. The automorphism group
$\Aut(G,\cdot)$, endowed with the compact-open topology induced from $C^0(G,G)$,
is a finite-dimensional real Lie group; see \cite{Hochschild1952aut,Dani1992}.

\begin{definition}
Let $(G,\cdot)$ be a connected Lie group. The \emph{affine group} associated
with $(G,\cdot)$ is the semidirect product
\[
\Aff(G,\cdot):=(G,\cdot)\rtimes \Aut(G,\cdot),
\]
where $\Aut(G,\cdot)$ acts on $G$ by evaluation.
\end{definition}

Now fix a Lie skew brace $(G,\cdot,\circ)$. Its \emph{lambda-action} is the map
\begin{equation}\label{lambda}
\lambda:(G,\circ)\longrightarrow \Aut(G,\cdot),\qquad a\longmapsto \lambda_a,
\end{equation}
where
\begin{equation}\label{lambdaspec}
\lambda_a(b):=a^{-1}\cdot(a\circ b),\qquad b\in G.
\end{equation}
It is well known that $\lambda$ is a smooth group homomorphism; see
\cite[Cor.~1.10]{Guarnieri2017} and \cite[Lemma~2.4]{DameleLoi2026}. Moreover,
the two group laws determine each other via
\[
a\circ b=a\cdot \lambda_a(b),
\qquad
a\cdot b=a\circ \lambda_{\overline a}(b),
\]
for all $a,b\in G$, where $\overline a$ denotes the inverse of $a$ in
$(G,\circ)$.

\subsection{Affine actions, post-Lie algebras, and integrability}\label{subsec:affine-postlie}

We now recall two complementary points of view that will be used repeatedly
throughout the paper. The first is the interpretation of Lie skew braces in
terms of simply transitive affine actions. The second is the infinitesimal
description in terms of post-Lie algebras and the corresponding integrability
problem.

\medskip
\noindent
\textbf{Affine actions.}

\begin{definition}\label{def:affine-action}
Let $(H,\circ)$ and $(K,\cdot)$ be connected real Lie groups. An
\emph{affine action} of $(H,\circ)$ on $(K,\cdot)$ is a Lie group homomorphism
\[
\rho:(H,\circ)\longrightarrow \Aff(K,\cdot)
=(K,\cdot)\rtimes \Aut(K,\cdot).
\]
If
\[
\rho(h)=(x,\varphi),
\]
the induced action of $h$ on $K$ is given by
\[
h\cdot k:=x\cdot \varphi(k),
\qquad k\in K.
\]
The action is said to be \emph{simply transitive} if for every $k\in K$ there
exists a unique $h\in H$ such that
\[
h\cdot e_K=k,
\]
or equivalently, for every $k_1,k_2\in K$ there exists a unique $h\in H$ such
that
\[
h\cdot k_1=k_2.
\]
\end{definition}

\begin{remark}\rm
If $(G,\cdot,\circ)$ is a Lie skew brace, then the map
\[
\iota:(G,\circ)\longrightarrow \Aff(G,\cdot),
\qquad
a\longmapsto (a,\lambda_a),
\]
is a Lie group homomorphism. The induced action on $(G,\cdot)$ is
\[
\iota(a)\cdot x
=
a\cdot \lambda_a(x)
=
a\circ x,
\qquad a,x\in G.
\]
Since left translations in the Lie group $(G,\circ)$ are simply transitive, it
follows immediately that this affine action is simply transitive.
\end{remark}

Conversely, every simply transitive affine action gives rise to a Lie skew
brace structure on the underlying manifold of the target Lie group. More
precisely, the following correspondence between connected Lie skew braces and
simply transitive affine actions was proved in
\cite[Thm.~2.13]{DameleLoi2026}.

\begin{theorem}\label{thm:LSB-affine-correspondence}
Let $(K,\cdot)$ be a connected real Lie group. The following are equivalent:
\begin{enumerate}[label=\textup{(\roman*)}]
    \item there exists a Lie group structure $\circ$ on the manifold $K$ such
    that $(K,\cdot,\circ)$ is a Lie skew brace;

    \item there exist a connected Lie group $(H,\circ)$ and a simply transitive
    affine action
    \[
    \rho:(H,\circ)\longrightarrow \Aff(K,\cdot).
    \]
\end{enumerate}
Moreover, when these conditions hold, the Lie group $(H,\circ)$ is isomorphic
to the multiplicative group of the corresponding Lie skew brace.
\end{theorem}

\begin{remark}\rm
Theorem~\ref{thm:LSB-affine-correspondence} shows that studying Lie skew brace
structures on a fixed connected Lie group $(K,\cdot)$ is equivalent to studying
connected Lie groups acting simply transitively on $(K,\cdot)$ by affine
transformations. In particular, many algebraic questions on Lie skew braces may
be reformulated as geometric questions on affine actions.
\end{remark}

\medskip
\noindent
\textbf{Post-Lie algebras and integrability.}
We briefly recall the Lie-theoretic passage from Lie groups to Lie algebras;
see, for instance, \cite{Hall2015,Warner}. Given a Lie group $(G,\cdot)$, its
Lie algebra is
\[
\Lie(G,\cdot)=(\g,[\, ,\,]_{\cdot}),
\]
where $\g=T_eG$ and the bracket is induced by left-invariant vector fields.
Conversely, every finite-dimensional real Lie algebra integrates to a connected
simply connected Lie group, unique up to isomorphism, and every other connected
Lie group with the same Lie algebra is obtained as a quotient by a discrete
central subgroup.

\begin{definition}\label{postLieAlgebras}
A \emph{post-Lie algebra} (PLA) is a quadruple
\[
(\g,[\, ,\,]_{\cdot},[\, ,\,]_{\circ},\triangleright),
\]
where $(\g,[\, ,\,]_{\cdot})$ and $(\g,[\, ,\,]_{\circ})$ are real Lie algebras
on the same finite-dimensional real vector space $\g$, and
\[
\triangleright:\g\times \g\to \g,
\qquad
(\xi,\eta)\mapsto \xi\triangleright \eta
\]
is a bilinear map such that, for all $\xi,\eta,\zeta\in\g$,
\begin{enumerate}[label=\textup{(\roman*)}]
\item
\[
[\xi,\eta]_{\circ}-[\xi,\eta]_{\cdot}
=
\xi\triangleright \eta-\eta\triangleright \xi;
\]
\item
\[
\xi\triangleright [\eta,\zeta]_{\cdot}
=
[\xi\triangleright \eta,\zeta]_{\cdot}
+
[\eta,\xi\triangleright \zeta]_{\cdot};
\]
\item
\[
[\xi,\eta]_{\circ}\triangleright \zeta
=
\xi\triangleright (\eta\triangleright \zeta)
-
\eta\triangleright (\xi\triangleright \zeta).
\]
\end{enumerate}
\end{definition}

\begin{remark}\rm
By condition~\textup{(i)} in Definition~\ref{postLieAlgebras}, the bracket
$[\, ,\,]_{\circ}$ is uniquely determined by $[\, ,\,]_{\cdot}$ and
$\triangleright$. Hence one may equivalently define a post-Lie algebra as a
triple
\[
(\g,[\, ,\,]_{\cdot},\triangleright),
\]
where $(\g,[\, ,\,]_{\cdot})$ is a real Lie algebra and $\triangleright$ is a
bilinear map satisfying
\begin{enumerate}[label=\textup{(\roman*)}]
\item
\[
\xi\triangleright [\eta,\zeta]_{\cdot}
=
[\xi\triangleright \eta,\zeta]_{\cdot}
+
[\eta,\xi\triangleright \zeta]_{\cdot};
\]
\item
\[
\bigl([\xi,\eta]_{\cdot}
+\xi\triangleright \eta
-\eta\triangleright \xi\bigr)\triangleright \zeta
=
\xi\triangleright (\eta\triangleright \zeta)
-
\eta\triangleright (\xi\triangleright \zeta).
\]
\end{enumerate}
One then defines the \emph{sub-adjacent Lie algebra}
$(\g,[\, ,\,]_{\circ})$ by
\[
[\xi,\eta]_{\circ}
:=
[\xi,\eta]_{\cdot}
+\xi\triangleright \eta
-\eta\triangleright \xi.
\]
See \cite{BaiGuoShengTang2023,Burde2012affineaction}.
\end{remark}

The category $\mathbf{PLA}$ has as objects post-Lie algebras
\[
(\g,[\, ,\,]_{\cdot},[\, ,\,]_{\circ},\triangleright),
\]
and as morphisms linear maps
\[
\phi:\g_1\to \g_2
\]
such that
\[
\phi([\xi,\eta]_{\cdot_1})
=
[\phi(\xi),\phi(\eta)]_{\cdot_2},
\qquad
\phi(\xi\triangleright_1 \eta)
=
\phi(\xi)\triangleright_2 \phi(\eta)
\]
for all $\xi,\eta\in\g_1$.

\medskip
Given a Lie skew brace $(G,\cdot,\circ)$, one has two Lie algebras
\[
\Lie(G,\cdot)=(\g,[\, ,\,]_{\cdot}),
\qquad
\Lie(G,\circ)=(\g,[\, ,\,]_{\circ}),
\]
on the same vector space $\g=T_eG$. The differential of the lambda-map
\eqref{lambda} at the identity,
\[
\lambda_{*e}:(\g,[\, ,\,]_{\circ})\longrightarrow
\Der(\g,[\, ,\,]_{\cdot}),
\]
defines a bilinear product on $\g$ by
\begin{equation}\label{trianglelambda}
\xi\triangleright \eta:=\lambda_{*e}(\xi)(\eta),
\qquad \xi,\eta\in\g.
\end{equation}
As shown in \cite{BaiGuoShengTang2023}, the quadruple
\[
(\g,[\, ,\,]_{\cdot},[\, ,\,]_{\circ},\triangleright)
\]
is a post-Lie algebra.

Thus one obtains a functor
\[
D:\mathbf{LSB}\to \mathbf{PLA},
\]
given on objects and morphisms by
\[
D(G,\cdot,\circ)
=
(\g,[\, ,\,]_{\cdot},[\, ,\,]_{\circ},\triangleright),
\qquad
D(\varphi)=\varphi_{*e}.
\]

\begin{definition}
A post-Lie algebra
\[
(\g,[\, ,\,]_{\cdot},[\, ,\,]_{\circ},\triangleright)
\]
is called \emph{integrable} if there exists a connected Lie skew brace
$(G,\cdot,\circ)$ such that
\[
D(G,\cdot,\circ)=
(\g,[\, ,\,]_{\cdot},[\, ,\,]_{\circ},\triangleright).
\]
\end{definition}

The following result may be viewed as a post-Lie/Lie-skew-brace analogue of
Lie's integration principle: under the simply connectedness assumption,
morphisms at the infinitesimal level integrate uniquely to global morphisms.
\begin{proposition}\label{prop:integration-morphisms-PLA}
Let
\[
(\mathfrak g_1,[\, ,\,]_{\cdot_1},[\, ,\,]_{\circ_1},\triangleright_1), \qquad
(\mathfrak g_2,[\, ,\,]_{\cdot_2},[\, ,\,]_{\circ_2},\triangleright_2)
\]
be integrable post--Lie algebras, and let
\[
(G_1,\cdot_1,\circ_1), \qquad (G_2,\cdot_2,\circ_2)
\]
be simply connected Lie skew braces integrating them. Let
\[
\varphi:\mathfrak g_1 \longrightarrow \mathfrak g_2
\]
be a morphism of post--Lie algebras. Then there exists a unique morphism of Lie skew braces
\[
F:(G_1,\cdot_1,\circ_1)\longrightarrow (G_2,\cdot_2,\circ_2)
\]
such that
\[
dF_e=\varphi.
\]
\end{proposition}

\begin{proof}
Since $\varphi$ preserves $[\,,\,]_{\cdot_1}$ and $[\,,\,]_{\circ_1}$, it integrates uniquely to Lie group
homomorphisms
\[
F:(G_1,\cdot_1)\to (G_2,\cdot_2),
\qquad
E:(G_1,\circ_1)\to (G_2,\circ_2),
\qquad
dF_e=dE_e=\varphi.
\]
We show that $F=E$. Let
\[
\lambda^{G_1}:(G_1,\circ_1)\to \mathrm{Aut}(G_1,\cdot_1),
\qquad
\lambda^{G_2}:(G_2,\circ_2)\to \mathrm{Aut}(G_2,\cdot_2)
\]
be the lambda-actions, and define
\[
\eta:\mathfrak g_1\to \mathrm{Der}(\mathfrak g_2,[\,,\,]_{\cdot_2}),
\qquad
\eta(\xi)(b):=\varphi(\xi)\triangleright_2 b .
\]
It is immediate that $\eta$ is a Lie algebra homomorphism, hence, since $(G_1,\circ_1)$ is simply
connected, it integrates uniquely to a Lie group homomorphism
\[
\widetilde{\lambda}:(G_1,\circ_1)\to \mathrm{Aut}(G_2,\cdot_2)
\qquad\text{with}\qquad
d\widetilde{\lambda}_e=\eta .
\]

Consider
\[
\Theta:(G_1,\circ_1)\to \mathrm{Aff}(G_2,\cdot_2),
\qquad
\Theta(g):=(F(g),\widetilde{\lambda}_g).
\]
Assume for the moment that $\Theta$ is a Lie group homomorphism. Let
\[
\iota_{G_2}:(G_2,\circ_2)\to \mathrm{Aff}(G_2,\cdot_2),
\qquad
\iota_{G_2}(u):=(u,\lambda^{G_2}_u).
\]
Then $\iota_{G_2}\circ E$ is also a Lie group homomorphism, and
\[
d\Theta_e=d(\iota_{G_2}\circ E)_e,
\]
because the first components are both $\varphi$, while on the second component
\[
d\widetilde{\lambda}_e=\eta=d\lambda^{G_2}_e\circ \varphi=d(\lambda^{G_2}\circ E)_e.
\]
By uniqueness of integration from the simply connected group $(G_1,\circ_1)$,
\[
\Theta=\iota_{G_2}\circ E.
\]
Comparing first components gives $F=E$, so $F$ is a homomorphism for both laws. It remains to prove that $\Theta$ is a homomorphism. For this it suffices to show that
\[
F\circ \lambda^{G_1}_g=\widetilde{\lambda}_g\circ F
\qquad\forall g\in G_1,
\]
for then
\[
\Theta(g)\Theta(h)
=
\bigl(F(g)\cdot_2 \widetilde{\lambda}_g(F(h)),\,\widetilde{\lambda}_g\widetilde{\lambda}_h\bigr)
=
\bigl(F(g\cdot_1 \lambda^{G_1}_g(h)),\,\widetilde{\lambda}_{g\circ_1 h}\bigr)
=
\Theta(g\circ_1 h).
\]

To prove the intertwining relation, it is enough to compare differentials at the identity. Set
\[
\rho_{G_1}:=D_{G_1}\circ \lambda^{G_1}:(G_1,\circ_1)\to \mathrm{Aut}(\mathfrak g_1,[\,,\,]_{\cdot_1}),
\qquad
\widetilde{\rho}:=D_{G_2}\circ \widetilde{\lambda}:(G_1,\circ_1)\to \mathrm{Aut}(\mathfrak g_2,[\,,\,]_{\cdot_2}),
\]
where $D_{G_1}(\alpha)=d\alpha_e$ and $D_{G_2}(\beta)=d\beta_e$. Then
\[
d(\rho_{G_1})_e(\xi)(\zeta)=\xi\triangleright_1 \zeta,
\qquad
d(\widetilde{\rho})_e(\xi)(b)=\varphi(\xi)\triangleright_2 b.
\]
Thus it is enough to prove that
\[
\varphi\circ \rho_{G_1}(g)=\widetilde{\rho}(g)\circ \varphi
\qquad\forall g\in G_1.
\]

Define
\[
\Psi(g)(T):=\widetilde{\rho}(g)\circ T\circ \rho_{G_1}(g)^{-1},
\qquad
T\in \mathrm{Hom}(\mathfrak g_1,\mathfrak g_2).
\]
Then $\Psi$ is a Lie group homomorphism. Hence the map
\[
f(g):=\Psi(g)(\varphi)
\]
is constant as soon as its differential at the identity vanishes. For $\xi,\zeta\in \mathfrak g_1$,
\[
\bigl(d\Psi_e(\xi)(\varphi)\bigr)(\zeta)
=
d(\widetilde{\rho})_e(\xi)(\varphi(\zeta))
-
\varphi\bigl(d(\rho_{G_1})_e(\xi)(\zeta)\bigr)
=
\varphi(\xi)\triangleright_2 \varphi(\zeta)-\varphi(\xi\triangleright_1 \zeta)=0,
\]
since $\varphi$ preserves the post-Lie product. Therefore $f$ is constant, and evaluating at $e$
gives
\[
\varphi\circ \rho_{G_1}(g)=\widetilde{\rho}(g)\circ \varphi
\qquad\forall g\in G_1.
\]

It follows that
\[
d(F\circ \lambda^{G_1}_g)_e=d(\widetilde{\lambda}_g\circ F)_e.
\]
Since $(G_1,\cdot_1)$ is simply connected, uniqueness of integration yields
\[
F\circ \lambda^{G_1}_g=\widetilde{\lambda}_g\circ F
\qquad\forall g\in G_1.
\]
Hence $\Theta$ is a Lie group homomorphism, and the conclusion follows.
Uniqueness is immediate from uniqueness of integration.
\end{proof}
\begin{corollary}
Let
$(\mathfrak g_1,[\, ,\,]_{\cdot_1},[\, ,\,]_{\circ_1},\triangleright_1)$
and
$(\mathfrak g_2,[\, ,\,]_{\cdot_2},[\, ,\,]_{\circ_2},\triangleright_2)$
be integrable post--Lie algebras. If they are isomorphic as post--Lie algebras, then the associated simply connected Lie skew braces
$(G_1,\cdot_1,\circ_1)$ and  $(G_2,\cdot_2,\circ_2)$
are isomorphic as Lie skew braces.
\end{corollary}

\begin{proof}
Let
$\varphi:\mathfrak g_1 \longrightarrow \mathfrak g_2$
be an isomorphism of post--Lie algebras. By the proposition, $\varphi$ integrates to a unique morphism of Lie skew braces
\[
F:(G_1,\cdot_1,\circ_1)\longrightarrow (G_2,\cdot_2,\circ_2)
\]
such that $dF_e=\varphi$. Applying the proposition to $\varphi^{-1}$, we obtain a morphism
\[
\widetilde F:(G_2,\cdot_2,\circ_2)\longrightarrow (G_1,\cdot_1,\circ_1)
\]
such that $d\widetilde F_e=\varphi^{-1}$. Hence
\[
d(\widetilde F\circ F)_e=\mathrm{id}_{\mathfrak g_1},
\qquad
d(F\circ \widetilde F)_e=\mathrm{id}_{\mathfrak g_2}.
\]
By uniqueness of integration, it follows that
\[
\widetilde F\circ F=\mathrm{id}_{G_1},
\qquad
F\circ \widetilde F=\mathrm{id}_{G_2}.
\]
Therefore $F$ is an isomorphism of Lie skew braces.
\end{proof}

\begin{corollary}
Let
\[
\mathrm{Int}:\mathbf{PLA}_{\mathrm{int}} \longrightarrow \mathbf{LSB}_{\mathrm{sc}}
\]
be the functor that assigns to an integrable post--Lie algebra its simply connected integrating Lie skew brace, and to a morphism the unique morphism integrating it. Then $\mathrm{Int}$ is an equivalence of categories.
\end{corollary}

\begin{proof}
By the proposition, $\mathrm{Int}$ is well defined and functorial. It is faithful, since morphisms between simply connected Lie skew braces are uniquely determined by their differential at the identity, and it is full, since every morphism of integrable post--Lie algebras integrates to a morphism of Lie skew braces.
It remains to prove that every simply connected Lie skew brace arises from an integrable post--Lie algebra. Let $(G,\cdot,\circ)$ be a simply connected Lie skew brace, and let
$(\mathfrak g,[\, ,\,]_\cdot,[\, ,\,]_\circ,\triangleright)$
be the associated post--Lie algebra. Then $\mathfrak g$ is integrable, and the simply connected Lie skew brace integrating it is isomorphic to $(G,\cdot,\circ)$.
Therefore $\mathrm{Int}$ is an equivalence of categories.
\end{proof}

The problem of determining which post-Lie algebras are integrable is highly
nontrivial and remains widely open; see, for example,
\cite{BaiGuoShengTang2023,Ebner2019postliealgebra,Burde2021crystallographicpostliealgebra,Burde2022rigiditypostliealgebra}.


\subsection{Ideals and simplicity}\label{subsec:ideals-simplicity}

We now recall the notion of ideal for a Lie skew brace, which is the natural
Lie-theoretic analogue of the corresponding notion for skew braces, and the
resulting definition of simplicity.

\begin{definition}\label{def:ideal-LSB}
Let $(G,\cdot,\circ)$ be a Lie skew brace. A subset $I\subseteq G$ is called an
\emph{ideal} of $(G,\cdot,\circ)$ if it is a connected closed Lie subgroup such
that
\begin{enumerate}[label=\textup{(\roman*)}]
    \item $I\trianglelefteq (G,\cdot)$;
    \item $I\trianglelefteq (G,\circ)$;
    \item $\lambda_a(I)=I$ for every $a\in G$.
\end{enumerate}
In this case we write
\[
I\trianglelefteq (G,\cdot,\circ).
\]
\end{definition}

\begin{example}\rm
Let $G_e$ be the identity component of the underlying manifold $G$. Then
$G_e\trianglelefteq (G,\cdot,\circ)$.
Indeed, $G_e$ is a connected closed normal Lie subgroup for both group laws,
and every $\lambda_a$ is a homeomorphism fixing the identity, hence preserves
$G_e$.
\end{example}

The notion of ideal is compatible with the Lie skew brace structure in the
expected way: quotients by ideals naturally inherit a Lie skew brace
structure, as the following proposition shows.

\begin{proposition}\label{prop:quotient-LSB}
Let $(G,\cdot,\circ)$ be a connected Lie skew brace and let
$I\trianglelefteq (G,\cdot,\circ)$. Then the quotient manifold $G/I$ carries a
natural Lie skew brace structure given by
\[
(aI)\cdot (bI):=(a\cdot b)I,
\qquad
(aI)\circ (bI):=(a\circ b)I,
\]
for all $a,b\in G$. Moreover, the quotient map
\[
\pi:G\to G/I
\]
is a morphism of Lie skew braces.
\end{proposition}
\begin{proof}
This follows by a straightforward verification. Since \(I\) is a connected
closed normal Lie subgroup of both \((G,\cdot)\) and \((G,\circ)\), the quotient
manifold \(G/I\) carries Lie group structures for both operations, defined by
\[
(aI)\cdot (bI):=(a\cdot b)I,
\qquad
(aI)\circ (bI):=(a\circ b)I,
\]
for all \(a,b\in G\). These operations are well defined because \(I\) is normal
in both groups.

We now verify the skew brace identity on \(G/I\). For \(a,b,c\in G\), one has
\begin{align*}
(aI)\circ\bigl((bI)\cdot(cI)\bigr)
&=(aI)\circ\bigl((b\cdot c)I\bigr)\\
&=\bigl(a\circ (b\cdot c)\bigr)I\\
&=\bigl((a\circ b)\cdot a^{-1}\cdot (a\circ c)\bigr)I\\
&=((a\circ b)I)\cdot (aI)^{-1}\cdot ((a\circ c)I),
\end{align*}
which coincides with
\[
\bigl((aI)\circ (bI)\bigr)\cdot (aI)^{-1}\cdot \bigl((aI)\circ (cI)\bigr).
\]
Thus \((G/I,\cdot,\circ)\) is a Lie skew brace.
Finally, the quotient map \(\pi:G\to G/I\) is smooth and satisfies
\[
\pi(a\cdot b)=\pi(a)\cdot \pi(b),
\qquad
\pi(a\circ b)=\pi(a)\circ \pi(b),
\]
for all \(a,b\in G\). Therefore \(\pi\) is a morphism of Lie skew braces.
\end{proof}

\begin{definition}\label{def:simple-LSB}
A connected Lie skew brace $(G,\cdot,\circ)$ is said to be \emph{simple} if
its only ideals are $\{e\}$ and $G$.
\end{definition}

\begin{remark}\label{rem:simple-group-implies-simple-LSB}\rm
If either $(G,\cdot)$ or $(G,\circ)$ is a simple Lie group (in the sense of
Definition~\ref{def:simple-lie-group}), then
$(G,\cdot,\circ)$ is a simple Lie skew brace. Indeed, every ideal of
$(G,\cdot,\circ)$ is, in particular, a connected closed normal Lie subgroup of
both $(G,\cdot)$ and $(G,\circ)$.
\end{remark}

\subsection{The simple case: trivial and almost trivial Lie skew braces}\label{subsec:simple-rigidity}
\begin{definition}\label{def:simple-lie-group}
A connected Lie group $G$ is called \emph{simple} if it is nonabelian and has no
nontrivial proper connected closed normal Lie subgroups.
\end{definition}

\begin{remark}\rm
For connected Lie groups, this is equivalent to requiring that $\Lie(G)$ be a
simple Lie algebra. This notion should not be confused with simplicity in the sense of
abstract group theory, since a connected simple Lie group may still
have nontrivial discrete normal subgroups.
\end{remark}

We now isolate the case in which the multiplicative  Lie group is
simple. In this situation, the Lie skew brace structure is extremely rigid: the
only possibilities are the trivial one and the opposite one.

\begin{definition}\label{def:trivial-almost-trivial}
A Lie skew brace $(G,\cdot,\circ)$ is said to be \emph{trivial} if
\[
a\circ b=a\cdot b
\qquad\text{for all }a,b\in G.
\]
It is said to be \emph{almost trivial} if
\[
a\circ b=b\cdot a
\qquad\text{for all }a,b\in G,
\]
that is,
\[
\circ=\cdot^{\mathrm{opp}}.
\]
\end{definition}

The following infinitesimal rigidity statement is essentially contained in
\cite[Prop.~4.6]{Burde2012affineaction}. We state it here in the form needed
for the proof of Theorem~\ref{thm:simple-group-rigidity}.

\begin{proposition}\label{prop:PLA-simple-rigidity}
Let
$(\g,[\, ,\,]_{\cdot},[\, ,\,]_{\circ},\triangleright)$
be a real post-Lie algebra. Assume that
$(\g,[\, ,\,]_{\circ})$
is simple. Then
$(\g,[\, ,\,]_{\cdot})$
is simple as well, and exactly one of the following holds:
\begin{enumerate}[label=\textup{(\roman*)},leftmargin=2.4em]
    \item
    \[
    \triangleright=0
    \qquad\text{and}\qquad
    [\, ,\,]_{\circ}=[\, ,\,]_{\cdot};
    \]
    \item
    \[
    \xi\triangleright \eta=-[\xi,\eta]_{\cdot}
    \qquad\text{for all }\xi,\eta\in\g,
    \]
    and
    \[
    [\, ,\,]_{\circ}=-[\, ,\,]_{\cdot}.
    \]
\end{enumerate}
\end{proposition}

\begin{proof}
This is the simple case of \cite[Proposition~4.6]{Burde2012affineaction}. The
argument given there is formulated over a field of characteristic \(0\), and
therefore applies in particular over \(\R\). The only additional input used in
the exclusion of the third case is Jacobson's fixed-point theorem; for the
required form over characteristic \(0\), see \cite[\S 5, Theorem~9]{Jac62}.
\end{proof}

We now pass from the infinitesimal statement to Lie skew braces.

\begin{theorem}\label{thm:simple-group-rigidity}
Let $(G,\cdot,\circ)$ be a connected Lie skew brace. Assume that
$(G,\circ)$ is a simple Lie group. Then $(G,\cdot)$ is a simple Lie group as
well, and
\[
\circ=\cdot
\qquad\text{or}\qquad
\circ=\cdot^{\mathrm{opp}}.
\]
Equivalently, $(G,\cdot,\circ)$ is either trivial or almost trivial.
\end{theorem}

\begin{proof}
Let
$\pi:\widetilde G\to G$
be the universal covering map. By \cite[Lemma~3.2]{DameleLoi2026}, the
universal covering manifold \(\widetilde G\) carries a lifted Lie skew brace
structure
$(\widetilde G,\widetilde\cdot,\widetilde\circ)$
such that \(\pi\) is a morphism for both group laws. In particular,
\((\widetilde G,\widetilde\circ)\) is connected and simply connected, and its
Lie algebra is naturally identified with
$\Lie(G,\circ)$.
Let
$D(\widetilde G,\widetilde\cdot,\widetilde\circ)
=
(\g,[\, ,\,]_{\widetilde\cdot},[\, ,\,]_{\widetilde\circ},\triangleright)$
be the associated post-Lie algebra. Since \((G,\circ)\) is a connected simple
Lie group, its Lie algebra
$(\g,[\, ,\,]_{\widetilde\circ})$
is simple. Therefore Proposition~\ref{prop:PLA-simple-rigidity} applies, and
shows in particular that
$(\g,[\, ,\,]_{\widetilde\cdot})$
is simple as well. Hence \((G,\cdot)\) is a connected simple Lie group.

We now distinguish the two cases given by
Proposition~\ref{prop:PLA-simple-rigidity}.

\medskip
\noindent
\textbf{Case 1.}
Assume that
\[
\triangleright=0
\qquad\text{and}\qquad
[\, ,\,]_{\widetilde\circ}=[\, ,\,]_{\widetilde\cdot}.
\]
By definition of the post-Lie product,
\[
\xi\triangleright \eta
=
(\widetilde\lambda_{*e}(\xi))(\eta),
\qquad \xi,\eta\in\g,
\]
where
\[
\widetilde\lambda:(\widetilde G,\widetilde\circ)\to
\Aut(\widetilde G,\widetilde\cdot)
\]
is the lifted lambda-action. Thus \(\triangleright=0\) means precisely that
$\widetilde\lambda_{*e}=0$.
Since \(\widetilde\lambda\) is a Lie group homomorphism and
\((\widetilde G,\widetilde\circ)\) is connected, it follows that
\(\widetilde\lambda\) is trivial. Therefore
\[
a\widetilde\circ b
=
a\widetilde\cdot \widetilde\lambda_a(b)
=
a\widetilde\cdot b
\qquad\text{for all }a,b\in \widetilde G,
\]
so
$\widetilde\circ=\widetilde\cdot.$
Since \(\pi\) is a homomorphism for both operations, we conclude that
\[
a\circ b=a\cdot b, \ \text{for all }a,b\in G.
\]

\medskip
\noindent
\textbf{Case 2.}
Assume that
\[
\xi\triangleright \eta=-[\xi,\eta]_{\widetilde\cdot}
\qquad\text{for all }\xi,\eta\in\g,
\]
and
\[
[\, ,\,]_{\widetilde\circ}=-[\, ,\,]_{\widetilde\cdot}.
\]
Then
\[
D(\widetilde G,\widetilde\cdot,\widetilde\circ)
=
(\g,[\, ,\,]_{\widetilde\cdot},- [\, ,\,]_{\widetilde\cdot},\triangleright).
\]
Consider now the opposite Lie group
\[
(\widetilde G,\widetilde\cdot^{\mathrm{opp}}),
\qquad
a\widetilde\cdot^{\mathrm{opp}} b:=b\widetilde\cdot a.
\]
Its Lie algebra is
\[
\Lie(\widetilde G,\widetilde\cdot^{\mathrm{opp}})
=
(\g,-[\, ,\,]_{\widetilde\cdot}).
\]
Moreover, the lambda-action of the Lie skew brace
\[
(\widetilde G,\widetilde\cdot,\widetilde\cdot^{\mathrm{opp}})
\]
is given by
\[
\lambda_a(b)
=
a^{-1}\widetilde\cdot
(a\widetilde\cdot^{\mathrm{opp}} b)
=
a^{-1}\widetilde\cdot (b\widetilde\cdot a),
\qquad a,b\in \widetilde G.
\]
Differentiating at the identity, one obtains
\[
(\lambda_{*e}(\xi))(\eta)=-[\xi,\eta]_{\widetilde\cdot}
\qquad\text{for all }\xi,\eta\in\g.
\]
Hence the associated post-Lie algebra is
\[
D(\widetilde G,\widetilde\cdot,\widetilde\cdot^{\mathrm{opp}})
=
(\g,[\, ,\,]_{\widetilde\cdot},- [\, ,\,]_{\widetilde\cdot},\triangleright).
\]
Therefore
$D(\widetilde G,\widetilde\cdot,\widetilde\circ)
=
D(\widetilde G,\widetilde\cdot,\widetilde\cdot^{\mathrm{opp}})$.
By Proposition~\ref{prop:integration-morphisms-PLA}, there exists a unique
morphism of Lie skew braces
\[
F:(\widetilde G,\widetilde\cdot,\widetilde\circ)\longrightarrow
(\widetilde G,\widetilde\cdot,\widetilde\cdot^{\mathrm{opp}})
\]
such that
$F_{*e}=\id_{\g}$.
In particular, \(F\) is a Lie group homomorphism
\[
F:(\widetilde G,\widetilde\cdot)\longrightarrow
(\widetilde G,\widetilde\cdot)
\]
whose differential at the identity is the identity on \(\g\). Since
\((\widetilde G,\widetilde\cdot)\) is connected and simply connected,
uniqueness of integration of Lie algebra morphisms implies that
$F=\id_{\widetilde G}$.
Since \(F\) is a morphism of Lie skew braces, for all \(a,b\in \widetilde G\)
one has
\[
F(a\widetilde\circ b)=F(a)\widetilde\cdot^{\mathrm{opp}}F(b).
\]
As \(F=\id_{\widetilde G}\), it follows that
\[
a\widetilde\circ b
=
a\widetilde\cdot^{\mathrm{opp}} b
=
b\widetilde\cdot a
\qquad\text{for all }a,b\in \widetilde G.
\]
Therefore
$\widetilde\circ=\widetilde\cdot^{\mathrm{opp}}$.
Passing to the quotient via \(\pi\), we obtain
\[
a\circ b=b\cdot a, \ \text{for all }a,b\in G,
\]
that is,
$\circ=\cdot^{\mathrm{opp}}$.
Thus \((G,\cdot,\circ)\) is almost trivial.
\end{proof}

\begin{remark}\label{rem:additive-simple-not-multiplicative-simple}\rm
The converse of Theorem~\ref{thm:simple-group-rigidity} is false in general.
Indeed, there exist connected Lie skew braces $(G,\cdot,\circ)$ for which the
additive Lie group $(G,\cdot)$ is simple, whereas the multiplicative Lie group
$(G,\circ)$ is solvable, hence not simple; see \cite{DameleLoi2026}.
\end{remark}

\begin{remark}\label{rem:simple-integrable-PLA}\rm
As an immediate infinitesimal consequence of
Theorem~\ref{thm:simple-group-rigidity}, every integrable real post-Lie algebra
whose circle Lie algebra is simple arises from either a trivial or an almost
trivial Lie skew brace.
\end{remark}

The previous rigidity result admits a natural extension to the compact case
without connectedness assumptions, linking the Lie-theoretic setting with the
finite case discussed in the introduction.

\begin{corollary}\label{cor:compact-abstractly-simple}
Let $(G,\cdot,\circ)$ be a compact Lie skew brace, not necessarily connected.
Assume that the multiplicative group $(G,\circ)$ is simple as an abstract
group. Then $(G,\cdot,\circ)$ is either trivial or almost trivial.
\end{corollary}

\begin{proof}
Let $G_e$ be the identity component of the underlying manifold $G$,
equivalently of the topological group $(G,\circ)$. Then
$G_e\trianglelefteq (G,\circ)$.
Since $(G,\circ)$ is simple as an abstract group, either
\[
G_e=G
\qquad\text{or}\qquad
G_e=\{e\}.
\]

If $G_e=G$, then $(G,\circ)$ is connected. Hence $(G,\circ)$ is a connected
simple Lie group, and the conclusion follows from
Theorem~\ref{thm:simple-group-rigidity}.
If $G_e=\{e\}$, then $G$ is totally disconnected. Since $G$ is compact, it is
finite. Thus $(G,\cdot,\circ)$ is a finite skew brace, and the conclusion
follows from the corresponding rigidity results for finite skew braces; see
\cite{Byott2004,Byott2024,BBEJP}.
\end{proof}


\section{Connected compact simple LSBs}\label{sec:compact-simple}

In the compact connected setting, simplicity is extremely restrictive and
forces the underlying Lie groups to be simple, except for the one-dimensional
torus. Before proving this rigidity result, it is worth stressing that this
phenomenon is genuinely specific to the compact connected Lie setting. 
In the noncompact case, there exist simple Lie skew braces for which
neither of the two underlying Lie groups is simple, and in fact both
may be solvable; see Example~\ref{ex:A1-1-model}. Similar phenomena
already occur in the finite skew brace setting; see
Remark~\ref{rem:finite-simple-skew-braces-solvable}.
By contrast, the compact connected Lie case is much more rigid, as we now
show.

\begin{theorem}\label{thm:main-compact-simple}
Let $(G,\cdot,\circ)$ be a compact connected Lie skew brace, and assume that it
is simple. Then exactly one of the following holds:
\begin{enumerate}[label=\textup{(\alph*)}]
    \item $(G,\cdot)=(G,\circ)=S^1$;
    \item $(G,\cdot)$ and $(G,\circ)$ are simple Lie groups, and
    \[
    \circ=\cdot
    \qquad\text{or}\qquad
    \circ=\cdot^{\mathrm{opp}}.
    \]
\end{enumerate}
\end{theorem}

In order to prove Theorem~\ref{thm:main-compact-simple}, we begin with a few
auxiliary results.

Let $(G,\cdot,\circ)$ be a Lie skew brace. For $x,y\in G$, the \emph{star
product} of $x$ and $y$ is defined by
\[
x*y=\lambda_x(y)\cdot y^{-1}.
\]
If $A,B\subseteq G$, we define
\[
A*B=\langle a*b \mid a\in A,\ b\in B\rangle_{\cdot},
\]
where $\langle S\rangle_{\cdot}$ denotes the subgroup generated by $S$ in the
group $(G,\cdot)$.

The following fact is well known for skew braces (see
\cite[Lemma~1.9]{CedoSmoktunowiczVendramin2018}), and its proof is purely
algebraic, so it extends verbatim to Lie skew braces.

\begin{lemma}\label{lem:char-star-ideal}
Let $(G,\cdot,\circ)$ be a connected Lie skew brace and let $I$ be a connected
closed characteristic subgroup of $(G,\cdot)$. If $I*G\leq I$, then
$I\trianglelefteq (G,\cdot,\circ)$.
\end{lemma}

The next lemma provides a useful criterion for producing ideals from
characteristic subgroups of the additive group.

\begin{lemma}\label{FundLemma}
Let $(G,\cdot,\circ)$ be a Lie skew brace and let $H$ be a connected closed
characteristic subgroup of $(G,\cdot)$. Then the map
\[
\lambda_H:(G,\circ)\longrightarrow \Aut\bigl((G,\cdot)/H\bigr),\qquad
g\longmapsto \bigl(xH\mapsto \lambda_g(x)H\bigr),
\]
is a well-defined Lie group homomorphism. Moreover,
\[
H\leq \ker(\lambda_H)
\iff
H\trianglelefteq (G,\cdot,\circ).
\]
\end{lemma}

\begin{proof}
Since $H$ is closed in $(G,\cdot)$, the quotient $(G,\cdot)/H$ is a Lie group.

We first show that $\lambda_H$ is well defined. If $xH=yH$, then
$xy^{-1}\in H$. Since $H$ is characteristic in $(G,\cdot)$ and
$\lambda_g\in\Aut(G,\cdot)$ for every $g\in G$, we get
$\lambda_g(xy^{-1})\in H$, hence
\[
\lambda_g(x)H=\lambda_g(y)H.
\]
Therefore $\lambda_H$ is well defined.

Since
\[
\lambda:(G,\circ)\to\Aut(G,\cdot)
\]
is a Lie group homomorphism, it follows immediately that $\lambda_H$ is also a
Lie group homomorphism.

Suppose now that $H\leq \ker(\lambda_H)$. Let $h\in H$ and $g\in G$. Then
\[
\lambda_h(g)H=gH,
\]
so $\lambda_h(g)\cdot g^{-1}\in H$. Therefore
\[
h*g=\lambda_h(g)\cdot g^{-1}\in H,
\]
and hence $H*G\subseteq H$. By Lemma~\ref{lem:char-star-ideal}, this implies
\[
H\trianglelefteq (G,\cdot,\circ).
\]

Conversely, assume that
\[
H\trianglelefteq (G,\cdot,\circ).
\]
Then for every $h\in H$ and $g\in G$ one has
\[
h*g=\lambda_h(g)\cdot g^{-1}\in H,
\]
so $\lambda_h(g)H=gH$. Thus $h\in\ker(\lambda_H)$, proving that
$H\leq \ker(\lambda_H)$.
\end{proof}

\begin{lemma}\label{lem:compact-abelian-trivial}
Let $(G,\cdot,\circ)$ be a connected compact Lie skew brace. Assume that
$(G,\cdot)$ is abelian. Then
\[
(G,\cdot)\cong \mathbb{T}^n
\]
for some $n\ge 0$, the lambda-action is trivial, and hence
\[
a\circ b=a\cdot b,\  \text{for all }a,b\in G.
\]
In particular, $(G,\cdot,\circ)$ is the trivial Lie skew brace on a torus.

Moreover, if $(G,\cdot,\circ)$ is simple, then necessarily
\[
(G,\cdot)=(G,\circ)=S^1.
\]
\end{lemma}

\begin{proof}
Since $(G,\cdot)$ is a connected compact abelian Lie group, there exists
$n\ge 0$ such that
\[
(G,\cdot)\cong \mathbb{T}^n.
\]

Consider the lambda-action
\[
\lambda:(G,\circ)\longrightarrow \Aut(G,\cdot).
\]
Since $(G,\cdot)\cong \mathbb{T}^n$, one has
\[
\Aut(G,\cdot)\cong \Aut(\mathbb{T}^n)\cong \GL_n(\mathbb{Z}),
\]
and $\GL_n(\mathbb{Z})$ is discrete. On the other hand, $(G,\circ)$ is
connected. Therefore the image $\lambda(G)$ is a connected subgroup of a
discrete group, hence it is trivial. Thus
\[
\lambda_a=\id_G
\qquad\text{for all }a\in G.
\]

Using the identity
\[
a\circ b=a\cdot \lambda_a(b),
\]
we obtain
\[
a\circ b=a\cdot b
\qquad\text{for all }a,b\in G.
\]
Hence the Lie skew brace is trivial, and in particular
\[
(G,\circ)=(G,\cdot)\cong \mathbb{T}^n.
\]

Assume now that $(G,\cdot,\circ)$ is simple. Since the brace is trivial, its
ideals are exactly the connected closed normal Lie subgroups of
$(G,\cdot)\cong \mathbb{T}^n$. But every connected closed subgroup of a torus
is again a torus, and if $n\ge 2$ there exist proper nontrivial connected
subtori. Therefore simplicity forces $n=1$. Hence
\[
(G,\cdot)=(G,\circ)=\mathbb{T}^1=S^1.
\]
\end{proof}

\begin{lemma}\label{lem:universal-cover-compact}
Let $(G,\cdot,\circ)$ be a Lie skew brace, and let
$(\widetilde G,\widetilde\cdot,\widetilde\circ)$
be its universal covering Lie skew brace. Assume that $\widetilde G$ is
compact. If $(G,\cdot,\circ)$ is simple, then
$(\widetilde G,\widetilde\cdot,\widetilde\circ)$
is simple.
\end{lemma}

\begin{proof}
Let
$\pi:\widetilde G\to G$
be the universal covering map. By \cite[Lemma~3.2]{DameleLoi2026}, $\pi$ is a
surjective homomorphism of Lie skew braces.
Let
$N\trianglelefteq (\widetilde G,\widetilde\cdot,\widetilde\circ)$
be an ideal. Since $N$ is closed in the compact manifold $\widetilde G$, it is
compact. Therefore $\pi(N)$ is compact, hence closed in $G$. Since $\pi$ is a
surjective homomorphism of Lie skew braces, the image $\pi(N)$ is an ideal of
$(G,\cdot,\circ)$.

Since $(G,\cdot,\circ)$ is simple, either
\[
\pi(N)=\{e\}
\qquad\text{or}\qquad
\pi(N)=G.
\]

If $\pi(N)=\{e\}$, then
\[
N\subseteq \ker(\pi).
\]
Since $\ker(\pi)$ is discrete and $N$ is connected, it follows that
$N=\{e\}$.

If $\pi(N)=G$, then
\[
d\pi_e\bigl(\Lie(N,\widetilde\cdot)\bigr)=\Lie(G,\cdot).
\]
Since
\[
d\pi_e:\Lie(\widetilde G,\widetilde\cdot)\to \Lie(G,\cdot)
\]
is an isomorphism, it follows that
\[
\Lie(N,\widetilde\cdot)=\Lie(\widetilde G,\widetilde\cdot).
\]
As $(N,\widetilde\cdot)$ and $(\widetilde G,\widetilde\cdot)$ are connected Lie
subgroups of $(\widetilde G,\widetilde\cdot)$ with the same Lie algebra, we
conclude that
\[
N=\widetilde G.
\]
Therefore $(\widetilde G,\widetilde\cdot,\widetilde\circ)$ is simple.
\end{proof}
\begin{proof}[Proof of Theorem~\ref{thm:main-compact-simple}]
We distinguish two cases.

\medskip
\noindent
\textbf{Case 1: \((G,\cdot)\) is not semisimple.}
Assume first that \((G,\cdot)\) is not semisimple. Since \((G,\cdot)\) is
compact and connected, it is reductive. Hence a compact connected
non-semisimple Lie group cannot be perfect, so its commutator subgroup
$H:=[G,G]_{\cdot}$
is a proper connected closed subgroup of \((G,\cdot)\).
Therefore the quotient \((G,\cdot)/H\) is a nontrivial compact connected
abelian Lie group, and so
\[
(G,\cdot)/H\cong \mathbb{T}^{n}
\qquad\text{for some }n\ge 1.
\]
By Lemma~\ref{FundLemma}, we obtain a continuous homomorphism
\[
\lambda_H:(G,\circ)\longrightarrow \Aut\bigl((G,\cdot)/H\bigr)
\cong \Aut(\mathbb{T}^{n})
\cong \GL_n(\mathbb{Z}).
\]
Since \((G,\circ)\) is connected and \(\GL_n(\mathbb{Z})\) is discrete, the
image of \(\lambda_H\) is trivial. Hence
\[
\ker(\lambda_H)=(G,\circ),
\]
and Lemma~\ref{FundLemma} yields
\[
H\trianglelefteq (G,\cdot,\circ).
\]
Since \((G,\cdot,\circ)\) is simple, it follows that
\[
H=\{e\}.
\]
Thus \((G,\cdot)\) is abelian. Since \(G\) is compact and connected,
Lemma~\ref{lem:compact-abelian-trivial} applies. Hence the Lie skew brace is
trivial and
\[
(G,\cdot)=(G,\circ)\cong \mathbb{T}^{n}
\]
for some \(n\ge 0\). Simplicity then forces \(n=1\). Therefore
\[
(G,\cdot)=(G,\circ)=S^{1},
\]
which proves \textup{(a)}.

\medskip
\noindent
\textbf{Case 2: \((G,\cdot)\) is semisimple.}

Assume now that \((G,\cdot)\) is semisimple. We proceed in several steps.

\medskip
\noindent
\textbf{Step 1: Reduction to the simply connected case.}
Let \(\widetilde G\) be the universal covering manifold of \(G\). By
\cite[Lemma~3.2]{DameleLoi2026}, \(\widetilde G\) carries a lifted Lie skew
brace structure
$(\widetilde G,\widetilde\cdot,\widetilde\circ)$
such that the covering map
$\pi:\widetilde G\to G$
is a morphism of Lie skew braces. Since \((G,\cdot)\) is compact and
semisimple, its fundamental group is finite; hence \(\pi\) is a finite
covering and \(\widetilde G\) is compact. By
Lemma~\ref{lem:universal-cover-compact}, simplicity passes to the universal
cover. Replacing \(G\) by \(\widetilde G\), we may therefore assume that both
\((G,\cdot)\) and \((G,\circ)\) are simply connected.

\medskip
\noindent
\textbf{Step 2: The lambda-action is inner.}
Since \((G,\circ)\) is connected, the subgroup \(\lambda(G)\leq \Aut(G,\cdot)\)
is connected. For a compact semisimple Lie group one has
\[
\Aut(G,\cdot)^0=\Inn(G,\cdot),
\]
because the outer automorphism group is finite. Therefore
\[
\lambda(G)\subseteq \Inn(G,\cdot).
\]

Let \(Z:=Z(G,\cdot)\). Since \((G,\cdot)\) is compact, semisimple and simply
connected, \(Z\) is finite and the natural quotient map
\[
q:(G,\cdot)\longrightarrow (G,\cdot)/Z\cong \Inn(G,\cdot)
\]
is a covering homomorphism of Lie groups. As \((G,\circ)\) is simply connected
and
\[
\lambda:(G,\circ)\to \Inn(G,\cdot)
\]
is a Lie group homomorphism, it lifts uniquely to a Lie group homomorphism
\[
\gamma:(G,\circ)\to (G,\cdot)
\]
such that
\[
q\circ \gamma=\lambda.
\]
Equivalently,
\begin{equation}\label{eq:lambda_is_conj_simple_sec_final}
\lambda_a=\Inn_{\gamma(a)}
\qquad (a\in G),
\end{equation}
that is,
\[
\lambda_a(x)=\gamma(a)\cdot x\cdot \gamma(a)^{-1}
\qquad \forall\, a,x\in G.
\]

\medskip
\noindent
\textbf{Step 3: A decomposition of \(\mathfrak f=\Lie((G,\cdot)\times(G,\cdot))\).}

Set
\[
K:=(G,\cdot)\times (G,\cdot),
\qquad
\Delta:=\{(g,g)\mid g\in G\},
\]
and let
\[
\mathfrak f:=\Lie(K),
\qquad
\mathfrak d:=\Lie(\Delta).
\]

Define
\begin{equation}\label{eq:Phi_def_simple_sec_final}
\Phi:(G,\circ)\longrightarrow K,
\qquad
\Phi(a):=\bigl(a\cdot \gamma(a),\,\gamma(a)\bigr).
\end{equation}
We first show that \(\Phi\) is a Lie group homomorphism. Let \(a,b\in G\). Since
\(\gamma:(G,\circ)\to (G,\cdot)\) is a Lie group homomorphism, one has
\[
\gamma(a\circ b)=\gamma(a)\gamma(b).
\]
Moreover, by \eqref{eq:lambda_is_conj_simple_sec_final},
\[
\lambda_a(b)=\gamma(a)b\gamma(a)^{-1}.
\]
Using
\[
a\circ b=a\cdot \lambda_a(b),
\]
we obtain
\[
a\circ b=a\gamma(a)b\gamma(a)^{-1}.
\]
Therefore
\begin{align*}
\Phi(a\circ b)
&=
\bigl((a\circ b)\gamma(a\circ b),\,\gamma(a\circ b)\bigr)\\
&=
\bigl(a\gamma(a)b\gamma(a)^{-1}\gamma(a)\gamma(b),\,\gamma(a)\gamma(b)\bigr)\\
&=
\bigl(a\gamma(a)b\gamma(b),\,\gamma(a)\gamma(b)\bigr)\\
&=
\bigl(a\gamma(a),\gamma(a)\bigr)\bigl(b\gamma(b),\gamma(b)\bigr)\\
&=
\Phi(a)\Phi(b).
\end{align*}
Thus \(\Phi\) is a homomorphism.
It is also injective. Indeed, if \(\Phi(a)=(e,e)\), then \(\gamma(a)=e\), and
hence \(a=e\). Thus \(\ker\Phi=\{e\}\). Since \((G,\circ)\) is compact and \(K\)
is Hausdorff, \(\Phi(G)\) is a closed Lie subgroup of \(K\), and
\[
\dim \Phi(G)=\dim G.
\]
Let
\[
\mathfrak g_\Phi:=\Lie(\Phi(G)).
\]

Now consider the multiplication map
\[
F:\Phi(G)\times \Delta\longrightarrow K,
\qquad
(g,\delta)\longmapsto g\delta.
\]
We claim that \(F\) is a diffeomorphism.
To prove surjectivity, let \((g_1,g_2)\in K\). Set
\[
a:=g_1g_2^{-1},
\qquad
u:=\gamma(a)^{-1}g_2.
\]
Then
\[
\Phi(a)(u,u)
=
(a\gamma(a),\gamma(a))(u,u)
=
(a\gamma(a)u,\gamma(a)u)
=
(g_1,g_2).
\]
Hence \(F\) is surjective.
We next show that
\[
\Phi(G)\cap \Delta=\{(e,e)\}.
\]
Indeed, if \(\Phi(a)\in\Delta\), then
\[
(a\gamma(a),\gamma(a))=(u,u)
\]
for some \(u\in G\). Hence \(a\gamma(a)=\gamma(a)\), so \(a=e\), and therefore
\(\Phi(a)=(e,e)\).

To prove injectivity, suppose
\[
g_1\delta_1=g_2\delta_2
\qquad
(g_1,g_2\in \Phi(G),\ \delta_1,\delta_2\in \Delta).
\]
Then
\[
g_2^{-1}g_1=\delta_2\delta_1^{-1}\in \Phi(G)\cap \Delta=\{(e,e)\},
\]
so \(g_1=g_2\) and \(\delta_1=\delta_2\). Thus \(F\) is injective.
Finally, we show that \(F\) is a local diffeomorphism. Since
\[
\dim(\Phi(G)\times \Delta)=\dim G+\dim G=\dim K,
\]
it is enough to prove that the differential at the identity is an isomorphism.
One has
\[
F_{*(e,e)}:\mathfrak g_\Phi\oplus \mathfrak d\to \mathfrak f,
\qquad
F_{*(e,e)}(X,Y)=X+Y.
\]
Because \(\Phi(G)\cap \Delta=\{(e,e)\}\), it follows that
\[
\mathfrak g_\Phi\cap \mathfrak d=\{0\},
\]
so \(F_{*(e,e)}\) is injective, hence bijective. By left translation, \(F\) is
a local diffeomorphism everywhere. Since it is also bijective, it is a global
diffeomorphism. Passing to Lie algebras, we obtain
\begin{equation}\label{eq:direct_sum_simple_sec_final}
\mathfrak f=\mathfrak g_\Phi\oplus \mathfrak d.
\end{equation}

\medskip
\noindent
\textbf{Step 4: Application of Ozeki's proposition.}
We apply \cite[Proposition~1]{Ozeki1977} to the compact Lie algebra
\(\mathfrak f\) together with the decomposition
\eqref{eq:direct_sum_simple_sec_final}. We obtain a decomposition
\[
\mathfrak f=I\oplus I'
\]
into ideals such that the projection
\[
p:\mathfrak f\to I
\]
restricts to a Lie algebra isomorphism
\begin{equation}\label{eq:p_iso_simple_sec_final}
p|_{\mathfrak g_\Phi}:\mathfrak g_\Phi\xrightarrow{\sim} I.
\end{equation}

\medskip
\noindent
\textbf{Step 5: The additive group \((G,\cdot)\) is simple.}
Assume, toward a contradiction, that \((G,\cdot)\) is not simple. Since it is
compact, semisimple and simply connected, there exists a decomposition
\[
(G,\cdot)\cong S_1\times \cdots \times S_r,
\qquad r\ge 2,
\]
where each \(S_i\) is a simply connected compact simple Lie group. Hence
\[
\mathfrak g_{\cdot}:=\Lie(G,\cdot)
=
\mathfrak s_1\oplus\cdots\oplus \mathfrak s_r,
\]
and
\[
\mathfrak f=\Lie(K)
=
\mathfrak g_{\cdot}\oplus \mathfrak g_{\cdot}
=
\bigoplus_{i=1}^r (\mathfrak s_i\oplus \mathfrak s_i).
\]
The simple ideals of \(\mathfrak f\) are precisely
\[
\mathfrak s_i\oplus 0
\qquad\text{and}\qquad
0\oplus \mathfrak s_i
\qquad (i=1,\dots,r).
\]

By Step~4, there exists an ideal \(I\trianglelefteq \mathfrak f\) such that
\[
p|_{\mathfrak g_\Phi}:\mathfrak g_\Phi\xrightarrow{\sim} I
\]
is a Lie algebra isomorphism. In particular,
\[
\dim I=\dim \mathfrak g_\Phi=\dim G=\dim \mathfrak g_{\cdot}
=\sum_{i=1}^r \dim \mathfrak s_i.
\]
Since \(I\neq 0\), it contains at least one simple ideal of \(\mathfrak f\);
fix one and denote it by
$I_{i_0}\subseteq I$.
Because \(r\ge 2\), one has
\[
\dim I>\dim I_{i_0},
\]
so \(I_{i_0}\) is a nonzero proper ideal of \(I\).

Let
\[
\mathfrak a:=(p|_{\mathfrak g_\Phi})^{-1}(I_{i_0}).
\]
Then \(\mathfrak a\) is a nonzero proper ideal of \(\mathfrak g_\Phi\).
Let \(A\trianglelefteq \Phi(G)\) be the connected closed normal subgroup with
Lie algebra \(\mathfrak a\), and set
\[
J:=\Phi^{-1}(A).
\]
Then \(J\) is a connected closed normal subgroup of \((G,\circ)\), and
\[
J\neq \{e\},
\qquad
J\neq G.
\]

Now consider the evaluation map
\[
\ev:K\longrightarrow (G,\cdot),
\qquad
\ev(u,v)=u\cdot v^{-1}.
\]
For every \(a\in G\),
\[
\ev(\Phi(a))=a.
\]
Hence the restriction
\[
\ev|_{\Phi(G)}:\Phi(G)\longrightarrow (G,\cdot)
\]
is a Lie group isomorphism, with inverse \(a\mapsto \Phi(a)\). Therefore
\[
(\ev|_{\Phi(G)})_{*(e,e)}:\mathfrak g_\Phi\longrightarrow \mathfrak g_{\cdot}
\]
is an isomorphism of Lie algebras.
Since \(\mathfrak a\) is a nonzero proper ideal of \(\mathfrak g_\Phi\), its
image
\[
(\ev|_{\Phi(G)})_{*(e,e)}(\mathfrak a)
\]
is a nonzero proper ideal of
\[
\mathfrak g_{\cdot}=\mathfrak s_1\oplus\cdots\oplus \mathfrak s_r.
\]
Hence there exists a nonempty proper subset
\[
\varnothing\neq \Lambda\subsetneq \{1,\dots,r\}
\]
such that
\[
(\ev|_{\Phi(G)})_{*(e,e)}(\mathfrak a)
=
\bigoplus_{i\in \Lambda}\mathfrak s_i.
\]
Let
\[
S_{\Lambda}:=\prod_{i\in \Lambda} S_i.
\]
Then \(S_{\Lambda}\) is a nontrivial proper connected closed normal subgroup of
\((G,\cdot)\), and
\[
\Lie(S_{\Lambda},\cdot)
=
\bigoplus_{i\in \Lambda}\mathfrak s_i
=
\Lie(J,\cdot).
\]
Since \(J\) and \(S_{\Lambda}\) are connected Lie subgroups of \((G,\cdot)\)
with the same Lie algebra, it follows that
\[
J=S_{\Lambda}.
\]

Thus \(J\) is a nontrivial proper connected closed normal subgroup of
\((G,\cdot)\), normal in \((G,\circ)\) by construction, and \(\lambda\)-stable
because, by \eqref{eq:lambda_is_conj_simple_sec_final}, each \(\lambda_a\) is
inner and therefore preserves every simple direct factor \(S_i\). Hence
$J\trianglelefteq (G,\cdot,\circ)$,
contradicting the simplicity of the Lie skew brace. This contradiction proves
that \((G,\cdot)\) is a simple Lie group.

\medskip
\noindent
\textbf{Step 6: Conclusion.}
Since \((G,\cdot)\) and \((G,\circ)\) are compact connected Lie groups on the
same underlying smooth manifold, they are homotopy equivalent. Hence
\[
\pi_n(G,\cdot)\cong \pi_n(G,\circ)
\qquad\text{for all }n\ge 1.
\]
By \cite[Thm.~2]{Boekholt1998}, compact connected Lie groups with isomorphic
homotopy groups in every degree are locally isomorphic. Therefore
\[
\Lie(G,\circ)\cong \Lie(G,\cdot).
\]
Since \((G,\cdot)\) is simple, its Lie algebra is simple. Hence
\(\Lie(G,\circ)\) is simple as well, and therefore \((G,\circ)\) is a simple
Lie group.

We may now apply Theorem~\ref{thm:simple-group-rigidity}. It follows that
\[
\circ=\cdot
\qquad\text{or}\qquad
\circ=\cdot^{\mathrm{opp}}.
\]
Thus \((G,\cdot,\circ)\) is either trivial or almost trivial. This proves
\textup{(b)} and completes the proof.
\end{proof}

\begin{remark}\rm
By Step~6, the compact connected Lie groups \((G,\cdot)\) and \((G,\circ)\)
are locally isomorphic. In particular, their universal covering Lie groups are
isomorphic.
At first sight, one might therefore expect that compact connected Lie skew
braces arise from the simply connected case by descending along discrete
central quotients. However, such a descent is not automatic in the Lie skew
brace setting: the discrete subgroup has to be compatible with both group
structures so that the brace structure descends to the quotient.
Theorem~\ref{thm:main-compact-simple} shows that, in the simple compact
connected case, this possible global freedom does not lead to new examples.
\end{remark}

The next corollary shows that, outside the exceptional case of $S^1$, the
notion of simplicity for compact connected Lie skew braces coincides with
simplicity of either underlying Lie group.

\begin{corollary}\label{cor:compact-simple-equivalence}
Let $(G,\cdot,\circ)$ be a compact connected Lie skew brace, and assume that
$G\neq S^1$.
Then the following are equivalent:
\begin{enumerate}[label=\textup{(\roman*)}]
    \item $(G,\cdot,\circ)$ is simple;
    \item $(G,\cdot)$ is a simple Lie group;
    \item $(G,\circ)$ is a simple Lie group.
\end{enumerate}
Moreover, if these conditions hold, then $(G,\cdot,\circ)$ is either trivial or
almost trivial.
\end{corollary}

\begin{proof}
If $(G,\cdot,\circ)$ is simple, then Theorem~\ref{thm:main-compact-simple}
shows that both $(G,\cdot)$ and $(G,\circ)$ are simple Lie groups, since
$G\neq S^1$. Hence \textup{(i)} implies \textup{(ii)} and \textup{(iii)}.
Conversely, if either $(G,\cdot)$ or $(G,\circ)$ is a simple Lie group, then
Remark~\ref{rem:simple-group-implies-simple-LSB} implies that
$(G,\cdot,\circ)$ is simple. Therefore Theorem~\ref{thm:main-compact-simple}
applies and yields
\[
\circ=\cdot
\qquad\text{or}\qquad
\circ=\cdot^{\mathrm{opp}},
\]
and in particular both $(G,\cdot)$ and $(G,\circ)$ are simple. Thus
\textup{(ii)} and \textup{(iii)} imply \textup{(i)}.
This proves the equivalence. The final statement is again a consequence of
Theorem~\ref{thm:main-compact-simple}.
\end{proof}

\begin{remark}\label{rem:finite-simple-skew-braces-solvable}\rm
The finite skew brace setting exhibits a much wider range of phenomena than the
compact connected Lie setting. Indeed, there exist finite simple skew braces
\((B,\cdot,\circ)\) whose additive group \((B,\cdot)\) is solvable, and even
abelian. For instance, Bachiller \cite{Bachiller2016Counterexample} exhibited a
simple skew brace such that
\[
(B,\cdot)\cong \mathbb Z_{3}\times (\mathbb Z_{2})^{3},
\qquad
(B,\circ)\cong S_{4}.
\]
In particular, simple skew braces may have abelian additive group. This example
also shows that, in the finite setting, solvability of the additive group does
not imply solvability of the skew brace itself.

More recently, Byott \cite{Byott2024simple} constructed an infinite family of
simple skew braces of order \(p^{p}q\), where \(p\) and \(q\) are primes such
that
\[
q \mid \frac{p^{p}-1}{p-1}.
\]
These examples are not braces and do not arise from nonabelian simple groups.
In particular, they provide further evidence that finite simple skew braces may
display a much richer behaviour than compact connected simple Lie skew braces.
Accordingly, the rigidity results proved above for compact connected Lie skew
braces have no direct analogue in the finite simple skew brace setting.
\end{remark}

\subsection{Solvable LSBs}\label{subsec:solvability}

We briefly discuss solvability for Lie skew braces, following the
commutator-ideal approach introduced for skew braces in \cite{BBEJP}. Since the
notion of ideal has already been fixed in the Lie setting, we only define the
associated derived series.

Let \((G,\cdot,\circ)\) be a Lie skew brace, and let \(I,J\) be ideals of
\(G\). We denote by \([I,J]_B\) the closed connected ideal generated by
\[
\{\brd{x}{y}:x\in I,\ y\in J\},\qquad
\{\brc{x}{y}:x\in I,\ y\in J\},\qquad
\{x*y:x\in I,\ y\in J\},
\]
where
\[
x*y=\lambda_x(y)\cdot y^{-1}.
\]
Observe that if \(I\) is an ideal, then \([I,I]_B\subseteq I\), since \(I\) is
a subgroup of both \((G,\cdot)\) and \((G,\circ)\), and is stable under
\(\lambda\). In particular, the sequence defined below is descending.

We define recursively
\[
\partial^0(G):=G,\qquad
\partial^{n+1}(G):=[\partial^n(G),\partial^n(G)]_B
\qquad (n\ge 0).
\]

\begin{definition}
A Lie skew brace \((G,\cdot,\circ)\) is said to be \emph{solvable} if
\[
\partial^n(G)=\{e\}
\]
for some \(n\ge 0\).
\end{definition}

This is the natural Lie-group analogue of the corresponding notion for skew
braces considered in \cite[Def.~16, Def.~18]{BBEJP}.

\begin{example}\rm
Consider \(G=\mathbb R^2\) with
\[
(x,s)\cdot (y,t)=(x+y,s+t),
\qquad
(x,s)\circ (y,t)=(x+e^{s}y,\ s+t).
\]
By the \((\mathrm{ab},\mathrm{solv})\) case in the proof of
\cite[Thm.~1.4]{DameleLoi2026}, this defines a Lie skew brace structure on
\(G\): the group \((G,\cdot)\) is abelian, while \((G,\circ)\) is isomorphic to
the connected affine group of the line, hence solvable and non-nilpotent. In
particular, the two underlying Lie groups are not isomorphic. It is not
difficult to see that this Lie skew brace \((G,\cdot,\circ)\) is solvable.
\end{example}

The next proposition is the Lie-group analogue of the corresponding fact for
skew braces. Since its proof is purely algebraic, it carries over verbatim to
the present setting.

\begin{proposition}\label{prop:solvable-LSB-implies-solvable-groups}
If $(G,\cdot,\circ)$ is a solvable Lie skew brace, then both underlying Lie
groups $(G,\cdot)$ and $(G,\circ)$ are solvable.
\end{proposition}

\begin{corollary}\label{cor:compact-solvable-LSB-torus}
Let \((G,\cdot,\circ)\) be a connected compact Lie skew brace which is
solvable. Then
\[
(G,\cdot)\cong (G,\circ)\cong \mathbb T^n
\]
for some \(n\ge 0\). In particular, the Lie skew brace is trivial.
\end{corollary}

\begin{proof}
By Proposition~\ref{prop:solvable-LSB-implies-solvable-groups}, both underlying
Lie groups \((G,\cdot)\) and \((G,\circ)\) are solvable. Since they are also
connected and compact, it is a standard fact that both are abelian. In
particular, \((G,\cdot)\) is abelian. Therefore
Lemma~\ref{lem:compact-abelian-trivial} applies. We conclude that
$(G,\cdot)\cong \mathbb{T}^n$
for some \(n\ge 0\), that the lambda-action is trivial, and hence
\[
a\circ b=a\cdot b, \ 
\text{for all }a,b\in G.
\]
Thus the Lie skew brace is trivial. In particular,
$(G,\cdot)\cong (G,\circ)\cong \mathbb{T}^n$.
\end{proof}

The converse of Proposition~\ref{prop:solvable-LSB-implies-solvable-groups} is
false, even in the connected Lie setting. Indeed, the following example
exhibits a connected noncompact \emph{simple} Lie skew brace
\((G,\cdot,\circ)\) whose additive and multiplicative Lie groups are both
solvable. In particular, solvability of the two underlying Lie groups does not
imply solvability of the Lie skew brace.

\begin{example}[A connected simple Lie skew brace with solvable underlying groups]\label{ex:A1-1-model}\rm
There exists a connected simply connected noncompact simple Lie skew brace
\((G,\cdot,\circ)\) such that both underlying Lie groups are solvable.
Let \(G=\R^3\) with additive law
\[
(X,Y,Z)\cdot(U,V,W)=(X+U,Y+V,Z+W).
\]
Define
\[
s(X,Y,Z):=X-YZ,
\]
and set
\begin{equation}\label{eq:circ-simple-solvable}
(X,Y,Z)\circ(U,V,W)=
\Big(
X+U+Ze^{s}V+Ye^{-s}W,\;
Y+e^{s}V,\;
Z+e^{-s}W
\Big).
\end{equation}

Then \((G,\circ)\) is a Lie group. Indeed, consider the global change of
coordinates
\[
\Phi:\R^3\longrightarrow \R^3,
\qquad
\Phi(X,Y,Z)=(x,y,z):=(X-YZ,\;Y,\;Z).
\]
Since \(x=s(X,Y,Z)\), a direct computation shows that, in the coordinates
\((x,y,z)\), the law \(\circ\) becomes
\begin{equation}\label{eq:circ-semidirect-form}
(x,y,z)\circ(x',y',z')
=
\bigl(
x+x',\;
y+e^{x}y',\;
z+e^{-x}z'
\bigr).
\end{equation}
Hence \((G,\circ)\) is isomorphic to the semidirect product
\[
(\R^2,+)\rtimes \R,
\]
where \(x\in \R\) acts on \((y,z)\in \R^2\) by
\[
(y,z)\longmapsto (e^x y,e^{-x}z).
\]
In particular, \((G,\circ)\) is connected, simply connected, and solvable.
Since \((G,\cdot)\cong (\R^3,+)\), the additive group is solvable as well.

The corresponding \(\lambda\)-maps are
\[
\lambda_{(X,Y,Z)}=
\begin{pmatrix}
1 & Ze^{s} & Ye^{-s}\\
0 & e^{s} & 0\\
0 & 0 & e^{-s}
\end{pmatrix},
\qquad s=X-YZ.
\]
One checks directly that
\[
\lambda_{p\circ q}=\lambda_p\lambda_q
\qquad\text{for all }p,q\in G,
\]
and therefore \((G,\cdot,\circ)\) is a Lie skew brace.

We now prove directly that \((G,\cdot,\circ)\) is simple.

Let \(I\trianglelefteq (G,\cdot,\circ)\) be an ideal. Since
\((G,\cdot)=(\R^3,+)\), every connected closed subgroup of \((G,\cdot)\) is a
real vector subspace. Thus \(I\) is a vector subspace of \(\R^3\).

Let \(e_1,e_2,e_3\) be the standard basis of \(T_0\R^3\), viewed in the
coordinates \((x,y,z)\) of \eqref{eq:circ-semidirect-form}. From that formula
one immediately obtains the Lie bracket of \((G,\circ)\):
\[
[e_1,e_2]_{\circ}=e_2,\qquad
[e_1,e_3]_{\circ}=-e_3,\qquad
[e_2,e_3]_{\circ}=0.
\]
Since \(I\) is normal in \((G,\circ)\), its tangent space at the identity,
which is again \(I\), must be an ideal of this Lie algebra.

We claim that the only proper nonzero ideals of
\(\Lie(G,\circ)\) are
\[
\R e_2,\qquad \R e_3,\qquad \R e_2\oplus \R e_3.
\]
Indeed, the derived algebra is
\[
[\Lie(G,\circ),\Lie(G,\circ)]
=
\R e_2\oplus \R e_3.
\]
Hence every proper ideal is contained in \(\R e_2\oplus \R e_3\). On this
plane, the operator \(\ad(e_1)\) is diagonal:
\[
\ad(e_1)|_{\R e_2\oplus \R e_3}
=
\begin{pmatrix}
1&0\\
0&-1
\end{pmatrix},
\]
so its only invariant subspaces are
\[
\{0\},\qquad \R e_2,\qquad \R e_3,\qquad \R e_2\oplus \R e_3.
\]
This proves the claim.

Therefore, if \(I\neq \{0\}\) and \(I\neq G\), then necessarily
\[
I=\R e_2,\qquad
I=\R e_3,
\qquad\text{or}\qquad
I=\R e_2\oplus \R e_3.
\]

We now use the \(\lambda\)-stability of ideals to exclude these three
possibilities. From the above formula for \(\lambda_{(X,Y,Z)}\), one gets
\[
\lambda_{(X,Y,Z)}(e_2)=Ze^{s}e_1+e^{s}e_2,
\qquad
\lambda_{(X,Y,Z)}(e_3)=Ye^{-s}e_1+e^{-s}e_3.
\]
Hence:
\begin{itemize}
    \item \(\R e_2\) is not \(\lambda\)-stable, since for \(Z\neq 0\),
    \[
    \lambda_{(X,Y,Z)}(e_2)\notin \R e_2;
    \]
    \item \(\R e_3\) is not \(\lambda\)-stable, since for \(Y\neq 0\),
    \[
    \lambda_{(X,Y,Z)}(e_3)\notin \R e_3;
    \]
    \item \(\R e_2\oplus \R e_3\) is not \(\lambda\)-stable, since for
    \(Z\neq 0\),
    \[
    \lambda_{(X,Y,Z)}(e_2)=Ze^{s}e_1+e^{s}e_2
    \notin \R e_2\oplus \R e_3.
    \]
\end{itemize}

Thus no proper nonzero vector subspace of \(\R^3\) can be an ideal of the Lie
skew brace. It follows that the only ideals are
\[
\{0\}\qquad\text{and}\qquad G.
\]
Therefore \((G,\cdot,\circ)\) is simple.
Finally, \((G,\cdot,\circ)\) is not solvable as a Lie skew brace. Indeed, it is
not trivial, since for example
\[
\lambda_{(1,0,0)}=
\begin{pmatrix}
1&0&0\\
0&e&0\\
0&0&e^{-1}
\end{pmatrix}
\neq \id.
\]
Hence \(\partial^1(G)\neq \{e\}\). Since \(\partial^1(G)\) is an ideal and the
brace is simple, one must have
\[
\partial^1(G)=G.
\]
Therefore \(\partial^n(G)=G\) for every \(n\ge 1\), so the Lie skew brace is
not solvable.
\end{example}



\end{document}